\documentclass{article}
\usepackage[utf8]{inputenc}
\usepackage{amssymb,amsmath,amsthm,amsfonts}
\usepackage{enumerate}
\usepackage{color}

  \newcommand{\R}{\ensuremath{\mathbf{R}}}%
  \newcommand{\N}{\ensuremath{\mathbf{N}}}%
 \newcommand{\mr}  {\mathring}%
 \newcommand{\bmr}[1]  {\overbrace{#1}^\circ}%
\newcommand{\m}  {\mathcal}%
\newcommand{\un}{\underline}%
  \newcommand{\M}  {\cal M}%
  \newcommand{\Z}{\ensuremath{\mathbf{Z}}}%
  \newcommand{\T}{\ensuremath{\mathbf{T}}}%
  \newcommand{\F}{\ensuremath{\mathbf{F}}}%
  \newcommand{\A}{\ensuremath{{\mathcal  A}}}
  \newcommand{\B}{\ensuremath{{\mathcal  B}}}
  \newcommand{\la}{\lambda}
  \newcommand{\eps}{\varepsilon}
\newcommand{\tens}{\otimes}
 \newcommand{\plB}[3]{\mathring{#1}_{#2_1}\otimes_{\m B}\cdots\otimes_{\m B}
\mathring{#1}_{#2_{#3}}}

\newtheorem{thm}{Theorem}[section]

\newtheorem{lemma}[thm]{Lemma}
\newtheorem{corollary}[thm]{Corollary}
\newtheorem{prop}[thm]{Proposition}

\theoremstyle{remark}

\newtheorem{rem}[thm]{Remark}

\theoremstyle{definition}

\author{Tao Mei\footnote{Research partially supported by the NSF grant DMS-1266042.} \ \ \ \ and \ \ \ \ \'Eric Ricard }

\title{Free Hilbert Transforms }

\begin{document}\maketitle

\begin{abstract} We study   analogues of classical Hilbert transforms as fourier multipliers on free groups.  We prove their  complete boundedness on non commutative $L^p$ spaces associated with the free group von Neumann algebras for all $1<p<\infty$. 
This implies that the decomposition of the free group $\F_\infty$ into reduced words starting with distinct    free generators  is completely unconditional in $L^p$. 
We study the case of  Voiculescu's amalgamated free products of von Neumann algebras as well. As by-products, we obtain a positive answer to a compactness-problem posed by Ozawa, a length independent estimate for Junge-Parcet-Xu's  free Rosenthal
 inequality, a Littlewood-Paley-Stein type inequality for geodesic paths   of free groups, and  a length reduction
 formula for $L^p$-norms of free group von Neumann algebras.

\end{abstract}
\section{Introduction}
Hilbert transform is a fundamental and influential object in the
mathematical analysis and signal processing. It was originally defined
for periodic functions. Given a trigonometric polynomial
$f(z)=\sum_{k=-N}^N c_k z^k$, let $P_+f=\sum_{k=1}^N c_k z^k$ be its
analytic part and $P_-f=\sum_{k=-N}^{-1} c_k z^k$ be its anti-analytic
part.  The Hilbert transform is formally defined as
$$H=-iP_++iP_-$$ and
clearly extends to an unitary on $L^2(\T)$. The case of $L^p$,
$1<p<\infty$ is more subtle. M. Riesz first proved that $H$ extends to
a bounded operator on $L^p(\T)$ for all $1<p<\infty$.  It is also well
known that $H$ is unbounded on $L^p(\T)$ at the end point $p=1,\infty$
but is of weak type (1,1). In modern harmonic analysis, the Hilbert
transform is considered as a basic example of Calder\'on-Zygmund
singular integral.  Its analogues have been studied in much more
general situations with connections to $L^p$-approximation, the
Hardy/BMO spaces, and more applied subjects.  

The Hilbert transforms
 appears also as the key tool to define conjugate functions in abstract
settings such as for Dirichlet algebras. In operator algebras, they show
up through Arveson's concept of maximal subdiagonal algebra of a von
Neumann algebra $\M$. Results about $L^p$-boundedness and weak-type
(1,1)-estimates in this situation were obtained by
N. Randrianantoanina in \cite{Ran}.

The object of this article is a natural analogue of the Hilbert
transform in the context of amalgamated free products of von Neumann
algebras.  The study is from a different view point to Arveson's and
is motivated from questions in the theory of $L^p$-Herz-Schur
multipliers on free groups.

Our model case is   the von Neumann algebra  $({\cal L}(\F_\infty),\tau)$ of free group  with  a countable set of  generators
$g_1,g_2,...$.
The associated $L^p$-spaces $L^p(\hat \F_\infty)$ is a non commutative
analogue of $L^p(\hat \Z)=L^p(\T)$.  Let ${\cal L}_{g_i^+} , {\cal
  L}_{g_i ^-}$  be the subsets of $\F_\infty$ of reduced words
starting respectively with $g_i, g_i^{-1}$.  One can naturally
associate to them projections; given a finitely supported function
$\hat x$ on $\F_\infty$, $\hat x=\sum_{g\in \F_\infty} c_g\delta_g,
c_g\in {\mathbb C}$, define $$L _{g_i^+}\hat x=\sum_{g\in {\cal
    L}_{g_i^+}}c_g\delta_g$$ and $L_{g_i^-}\hat x$ similarly. All of them
obviously extend to norm 1 projections on
$\ell_2(\F_\infty)=L^2(\hat\F_\infty)$. A natural question is whether these
projections are bounded on $L^p(\hat\F_\infty)$ and whether the
decomposition $\F_\infty=\{e\}\cup_{i\in \N,\eps\in \pm} {\cal L}_{g_k^\eps} $ is
unconditional in $L^p(\hat\F_\infty)$. To that purpose, we define a free
analogue of the classical Hilbert transform as the following map
\begin{eqnarray}
\label{H}
H_\varepsilon =  \varepsilon_1L_{g_1^+}+ \varepsilon_{-1} L_{g_1^-}+ \varepsilon_2 L_{g_2^+}+  \varepsilon_{-2} L_{g_2^-}+...
\end{eqnarray}
for $\varepsilon_i=\pm1$. We are interested in the (complete)
boundedness of $H_\varepsilon$ on $L^p(\hat \F_\infty)$ as well as possible
connections to semigroup-Hardy/BMO spaces and the $L^p$-approximation
property in the non commutative setting.  The question of the
$L^p(\hat \F_\infty)$-boundedness of $H_\varepsilon$ has been around for
some time. The authors learned from G. Pisier that  P. Biane asked this question and discussed with him around 2000. N. Ozawa pointed out  that the
$L_4(\hat \F_\infty)$-boundedness of $H_ \varepsilon$  answers
positively the problem he posed  at the end of  \cite{O10}. Junge-Parcet-Xu obtained some length dependent results for related questions in their work of Rosenthal's inequality  for amalgamated free products  (\cite{JPX07}).

  The first result (Theorem \ref{main}) of this article is a positive
 answer to the $L^p$-boundedness question of $H_\varepsilon$ in the
 general case of Voiculescu's amalgamated free products, which
 includes the free group of countable many generators as a particular case (Theorem \ref{main2}).
 
 One can also consider two similar Hilbert transforms. One is $$H_\varepsilon^{Ld}=\eps_e P_{d-1}+\sum_{h, |h|=d} \varepsilon_h L_h$$ with $P_{d}$ the projection onto reduced words with length $\leq d$ and $L_h$'s   the projections onto reduced words starting
 with $h$.   Another is 
 $$H_\varepsilon^{(d)}=\eps_e P_{d-1}+\sum_{g, |g|=1} \varepsilon_g
 L_g^{(d)}$$ with $L_g^{(d)}$'s the projections onto reduced words
 having $g$ as its $d$-th letter. Their (complete) boundedness on
 $L^p(\hat \F_n)$ can be easily deduced from that of $H_\varepsilon$
 with constants depending on $n$.  The main result of this article
 (Theorem \ref{hn}) says that $H^{(d)}_\varepsilon$'s are completely
 bounded on $L_p(\hat \F_\infty)$ for any $d\geq 1$.  While
 $H_\eps^{Ld}$'s are bounded for all $1<p<\infty$ but not completely
 bounded on $L_p(\hat \F_\infty),$ for any $p\neq 2, d\geq 2$.  The
 authors also prove a length reduction formula to compute $L^p$-norms
 and a Rosenthal inequality with length independent constants. 
 
 A classical argument, in proving the $L^p$-boundedness of the Hilbert transform $H$ is to use the following Cotlar's identity 
\begin{eqnarray}\label{Cotlar}
|H(f)|^2=|f|^2+H(\bar f Hf+\overline{Hf}f),
\end{eqnarray}
that allows to get the result for $L_{2p}$ from that of $L^p$ and
implies optimal estimates.  This identity holds in a general setting,
if one can identify a suitable ``analytic" algebra and defines the
corresponding Hilbert transform as the subtraction of two projections
on this algebra and its adjoint.  This is the case of non commutative
Hilbert transforms associated with Arveson's maximal subdiagonal
algebras (see Lemma 8.5 of \cite{PX03}).\footnote{Arveson's
  ``analytic" subalgebras do not seem available for amalgamated free
  products of von Neumann algebras in general. They are available for
  free group von Neumann algebras but the corresponding Hilbert
  transforms are different from ours and their formulations as
  Herz-Schur multipliers are difficult to determine. }  After obtaining
an initial proof of Theorem \ref{main2}, we observed that a free
version of Cotlar's identity (see (\ref{cotlar})) holds in the context of amalgamated free
products for $H_\eps$ with $|\varepsilon_k|\leq 1$.\footnote{The
  classical Cotlar's formula fails for $H=-iP_++\varepsilon iP_-$ if
  $\varepsilon\neq\pm1$.}  We were slightly surprised when this
observation came out, given that $H_\varepsilon$, defined in
(\ref{H}), is associated to subsets instead of subalgebras. On the
other hand, once it draws our attention, the proof of the identity and
Theorem \ref{main2} are not hard.  It is a surprise that  this identity 
was not noticed earlier.


We will introduce the notations and necessary preliminaries in Section
2.  The  Cotlar's formula for amalgamated free
products and   Theorem \ref{main} are proved in Section 3.1.  
Section 3.2 includes a few immediate consequences. Section 3.3 obtains a length  independent   Rosenthal inequality, which was initially proved by Junge/Parcet/Xu (Theorem A, \cite{JPX07}) restricted to   a fixed length. 
Section 4.1 proves our main result Theorem \ref{hn}.    Corollary \ref{PP} of that section
gives a length reduction formula and generalizes the main result of \cite {PP05}. Corollary \ref{Ozawa} (iii) answers positively the problem that Ozawa posed at the end of \cite{O10}. Section 4.3 studies   Littlewood-Paley-Stein type inequalities. Corollary \ref{branch} shows that the projection onto a geodesic path of the free group is completely bounded on $L^p$ for $1<p<\infty$.
 Theorem \ref{llpp} is a dyadic Littlewood-Paley-Stein
inequality for geodesic paths of free groups.

 
\section{Notations and preliminaries}

We refer the reader to \cite{VDN} and \cite{JPX07} for the definition
of amalgamated free products, and to \cite{PX03} and the references
therein for a formal definition and basic properties on non
commutative $L^p$ spaces.  For simplicity, we will restrict to the
case of finite von Neumann algebras but all the arguments should be
easily adapted to type III algebras with n.f. states.

 About noncommutative $L^p$-spaces associated to a finite von Neumann
 algebra $(\A,\tau)$, we will mainly need duality, interpolation
 and the non commutative Khintchine inequality (\cite{L86},
 \cite{LP91}) in $L^p(\A)$ as well as $p$-row and $p$-column spaces.
As usual we denote by $c_k=e_{k,1}$ and $r_k=e_{1,k}$, $d_k=e_{k,k}$ the canonical basis of the column, row and diagonal subspaces of the Schatten $p$-class $S_p(\ell_2(\N))$.

We will use the duality  $\langle x,y\rangle_{L^p,L^q}=\tau(xy)$ to identify 
$L^q(\A)$ with $L^p(\A)^*$ isometrically for $1\leq p<\infty$. At operator space level
this gives a complete isometry $L^p(A)^*=L^q(\A)^{op}$, see \cite{P98}.

 As $\A=L^\infty(\A)$ is finite,  the obvious embedding $L^\infty(\A)\subset L^1(\A)$ 
makes $(L^\infty(\A),L^1(\A))$ a compatible couple of Banach spaces. For $1< p<\infty$, the complex interpolation spaces between $\A$ and $L^1(\A)$ with index $\frac 1p$ is isometric to $L^p(\A)$ : 
\begin{eqnarray}
(L^\infty(\A),L^1(\A))_{\frac1p}&=&L^p(\A).
\end{eqnarray}

For a sequence $(x_k)$ in $L^p(\A)$, we use the classical notation 
$$\|(x_k)\|_{L^p(\A, \ell_2^{c})}=\|(\sum_k|x_k|^2)^\frac12\|_p,\quad \|(x_k)\|_{L^p(\A, \ell_2^{r})}=\|(\sum_k|x^*_k|^2)^\frac12\|_p,$$ and 
$$\|(x_k)\|_{L^p(\A, \ell_2^{cr})}=\left\{\begin{array}{ll} \max\{
\|(x_k)\|_{L^p(\A, \ell_2^{c})}, \|(x^*_k)\|_{L^p(\A, \ell_2^{c})}\}&
\textrm{if } 2\leq p\leq \infty\\ \inf_{ y_k+z_k=x_k}
\|(y_k)\|_{L^p(\A, \ell_2^{c})}+ \|(z^*_k)\|_{L^p(\A, \ell_2^{c})}&
\textrm{if } 0<p<2\end{array}\right..$$ We refer to \cite{P98} for non
commutative vector-valued $L^p$-spaces. The above definition is justified by the non commutative Khintchine inequalities:

\begin{lemma} (\cite{L86}, \cite{LP91},\cite{HM07}) Let $\varepsilon_k$
be independent Rademacher random variables, then for $1\leq p<\infty$, 
\begin{eqnarray}
\alpha_pE_\eps\|\sum_k \varepsilon_k\otimes x_k\|_p&\leq& \|(x_k)\|_{L^p(\A, \ell_2^{cr})}\leq \beta_p E_\eps\|\sum_k \varepsilon_k\otimes x_k\|_p. \label{khin}
\end{eqnarray}
\end{lemma}

Here $\varepsilon_k$ can also be replaced by  other orthonormal sequences of
some $L^2(\Omega, \mu)$, e.g. $z^{2^k}$ on the unit circle or standard
Gaussian.  For $z^{2^k}$ on the unit circle or standard
Gaussian, the best constant $\beta_p$
is $ \sqrt 2$ for $p=1$ and is $1$ for $p\geq 2$ (see
\cite{HM07}). $\alpha_p$ is $1$ for $1\leq p\leq 2$ and is of order
$\sqrt p$ as $p\rightarrow \infty$. (\ref{khin}) was pushed further to
the case of $0<p<1$ by Pisier and the second author recently (see \cite{PR14}).

\medskip

 If $(\A_k,\tau_k), k\geq 1$ are finite von Neumann algebras with
 a common sub-von Neumann algebra $(\B,\tau)$ with conditional
 expectation $E$ so that $\tau_k E=\tau$, we denote by
 $(\A,\tau)=(*_{\B}\A_k,\tau_k)$ the amalgamated free product of
 $(\A_k, \tau_k)$'s over $B$. We will briefly recall the construction
 to fix notation.

For any $x\in \A_k$, we denote by $\mathring{x}=x-E x$ and $\mathring \A_k=\{\mathring{x}; x\in \A_k\}$; there is a natural decomposition $\A_k=\B\oplus \mathring \A_k$.

The space
$$\m W =\B \oplus_{n\geq 1} \bigoplus_{\substack{(i_1,...,i_n)\in
    \N^n \\ i_1\neq i_2...\neq i_n}}\plB {\m A} i n=\oplus_{n\geq
  0}\bigoplus_{\substack{(i_1,...,i_n)\in
    \N^n \\ i_1\neq i_2...\neq i_n}} \m W_{ \un i}$$ is a $*$-algebra using concatenation and centering
with respect to $\B$. The natural projection $E$ onto $\B$ is a
conditional expectation and $\tau E$ is a trace on $\m W$ still
denoted by $\tau$. Then $(\A,\tau)$ is the finite von Neumann algebra
obtained by the GNS construction from $(\m W,\tau)$. Thus $\m W$ is
weak-$*$ dense in $\m A$ and dense in $L^p(\A)$ for $p<\infty$.

For multi-indices, we write $(i_1,...,i_n)=\un i\preceq^L \un j=(j_1,...,j_m)$ if
$m\geq n$ and $i_k=j_k$ for $k\leq n$ and $\un i\preceq^R \un
j=(j_1,...,j_m)$ if $m\geq n$ and $i_{m-k+1}=j_{m-k+1}$ for $1\leq
k\leq n$.  We also put $\un i\prec ^L \un j$ if $\un i\preceq^L \un j$
and $n<m$. We extend those relations for non zero elementary tensors
$g\in \m W_{\un i}$ and $h\in \m W_{\un j}$, we write $g\prec^L h$ if
$\un i\prec^L \un j$ and $g\prec^R h$ if $\un i\prec^R\un j$.

For $k\in \N$, put \begin{eqnarray*}\m L_k = \oplus_{k\preceq^L \un i}
  \m W_{\un i},\qquad \textrm{and}\qquad \m R_k = \oplus_{k\preceq^R \un i} \m W_{\un i}.\end{eqnarray*}
We denote the
associated orthogonal projections on $\m W$  by
$L_k$ and $R_k$. We use the convention $L_0=E$.

Given a sequence of $\varepsilon_k\in \B, k\in \N$, and $x\in \m W$, we let
$$H_\varepsilon(x) =\varepsilon_0 E(x)+ \sum_{ k\in \N} \varepsilon_k L_{ k}(x)
;\qquad H_\varepsilon^{op}= E(x)\varepsilon_0^*+ \sum_{k\in \N}
R_{k}(x)\varepsilon_k^*.$$ The main theorem is that, for $1<p<\infty$, $H_\varepsilon$ extends to $L^p$ 
and for any $x\in
L^p(\A)$, 
 \[  \|H_\varepsilon x\|_{p }\simeq^{ c_{p}} \| x\|_{p} ,\]
 for any choice of unitaries $\varepsilon_k\in \m Z(\m B)$ in the center of $\B$ and $1<p<\infty$.



\section{Amalgamated Free products }
\subsection{The   Cotlar  formula for free products}

We start with very basic observations, recall that $\mathring x=x-E x$ for $x\in \A$.
  \begin{prop}\label{ghid1}   For  $g\in \m W$, and $\varepsilon,\varepsilon'$ sequences in $\B$ 
\begin{enumerate}[(i)]
   \item  $H_\varepsilon(g^{*})=(H_\varepsilon^{op}(g))^{*}$.
\item $H_\varepsilon(\mr g)=\bmr{H_\varepsilon(g)}$.
 \item  $H_\varepsilon H_{\varepsilon'}^{op}( g)=H_{\varepsilon'}^{op} H_\varepsilon(g)$.
 \end{enumerate}
\end{prop}
\begin{proof}
This is clear on elementary tensors.
\end{proof}

We now give the free version of Cotlar's identity.
\begin{prop}\label{ghid}   For elementary tensors $g,h \in \m W,$
\begin{enumerate}[(i)]
\setcounter{enumi}{3}
      \item  $\bmr{H_\varepsilon(g^*h)}=\bmr{H_\varepsilon(g^*)h}$ if 
$g\nprec^L h$
\item  $\bmr{H_\varepsilon^{op}(g^*h)}=\bmr{g^*H_\varepsilon^{op}(h)}$ if 
$h\nprec^R g$. 

And for any $g,h \in \m W,$
  \item  $\bmr{H_\varepsilon(g^*)H_{\varepsilon'}^{op}(h)}=\bmr{H_\varepsilon(g^*H_{\varepsilon'}^{op}(h))}+\bmr{H_{\varepsilon'}^{op}(H_\varepsilon(g^*)h)}-\bmr{H^{op}_{\varepsilon'} H_\varepsilon(g^*h)}$.
\end{enumerate}
\end{prop}
\begin{proof}
 Let $g=g_1\otimes ...\otimes g_n\in \m W_{\un i}$ and $h=h_1\otimes
 ...\otimes h_m\in \m W_{\un j}$ with $\un i=(i_1,...,i_n)$ and $\un
 j=(i_1,...,j_m)$, $n,m\geq 0$. We start by proving $(iv)$ by
 induction on $n+m$.

If $n+m=0$, this is clear as $\bmr{H_\varepsilon(g^*h)}=\bmr{H_\varepsilon(g^*)h}=0.$

Assume $n+m\geq 1$ and $g\nprec^L h$. Note that necessarily $n\geq 1$.\\ 
First case: $i_1\neq j_1$ or $m=0$, then 
$$g^*h= g_n^* \otimes ...\otimes g_2^* \otimes g_1^*\otimes h_1\otimes
h_2 \otimes...\otimes h_m,$$
and $H_\varepsilon(g^*h)=H_\varepsilon(g^*)h=\varepsilon_{i_n} g^*h$.\\
Second case: $i_1= j_1$,
$$g^*h= g_n^* \otimes ...\otimes g_2^* \otimes (\bmr{g_1^*h_1})\otimes
h_2 \otimes...\otimes h_m+ (g_n^* \otimes ...\otimes g_2^*).(
(E{g_1^*h_1}) h_2 \otimes...\otimes h_m)$$ Put $\tilde
g=\bmr{h_1^*g_1}\otimes...\otimes g_n$, $\tilde h=h_2\otimes...\otimes
h_m$ and $\hat g=g_2\otimes...\otimes g_n$, $\hat h=(E{g_1^*h_1})
h_2\otimes...\otimes h_m$ (if $n=1$, $\hat g=1$) . Note that $\tilde g
\nprec^L \tilde h$ (or $\tilde g=0$) and $\hat g \nprec^L \hat h$ and the sum of their
length is strictly smaller than $n+m$. We can apply the formula to
them to get $\bmr{H_\varepsilon({\tilde g}^*\tilde
  h)}=\bmr{H_\varepsilon(\tilde g^*)\tilde
  h}=\varepsilon_{i_n}{{\tilde g}^*\tilde h}$ and
$\bmr{H_\varepsilon({\hat g}^*\hat h)}=\bmr{H_\varepsilon(\hat
  g^*)\hat h}=\varepsilon_{i_n}\bmr{{\hat g}^*\hat h}$ (this holds if
$n=1$ because then $m=1$ and $\bmr{{\hat g}^*\hat h}=0$). Finally
$$\bmr{H_\varepsilon(g^*h)}=\varepsilon_{i_n}({{\tilde g}^*\tilde h} +\bmr{{\hat g}^*\hat h})=\varepsilon_{i_n} \bmr{g^*h}=\bmr{H_\varepsilon(g^*)h}.$$

$(v)$ follows from $(iv)$ by taking adjoints.

To get $(vi)$ it suffices to do it for elementary tensors by linearity. Assume first that $g\nprec^Lh $, then obviously 
$g\nprec^L H_{\varepsilon'}^{op}(h)$, so by $(iv)$
$$\bmr{H_{\varepsilon}(g^*H_{\varepsilon'}^{op}(h))}=\bmr{H_\varepsilon(g^*)H_{\varepsilon'}^{op}(h)}, \qquad\bmr{H_\varepsilon(g^*)h}=\bmr{H_\varepsilon(g^*h)}.$$
Since the centering operation commutes with $H_{\varepsilon'}^{op}$ by Proposition 
\ref{ghid1}, we get $(vi)$.

If $g\prec^Lh$ then $h\nprec^R g$ and we can use $(v)$ and Proposition \ref{ghid1} $(ii)$ as above and $(iii)$ to get $(vi)$ as
$$\bmr{H_{\varepsilon'}^{op}(H_\varepsilon(g^*)h)}=\bmr{H_\varepsilon(g^*)H_{\varepsilon'}^{op}(h)}, \,
\bmr{H_{\varepsilon'}^{op}H_\varepsilon(g^*h)}=\bmr{H_\varepsilon H_{\varepsilon'}^{op}(g^*h)}=\bmr{H_\varepsilon(g^*H_{\varepsilon'}^{op}(h))}.$$  
\end{proof}
\begin{rem}
Removing the centering, we have obtained a Cotlar's formula  for $x=\sum_i g_i, y=\sum_j h_j, g_i,h_j\in \m W$
as follows,
\begin{eqnarray}
 && H_\varepsilon x(H_{\varepsilon'} y)^* -E[(H_{\eps_0}x-\eps_0 x)(H_{\varepsilon'} y-\eps'_0y)^*]\nonumber\\&=& H_\varepsilon
  (xH^{op}_{\varepsilon'}(y^*))+ H_{\varepsilon'}
  ^{op}(H_\varepsilon (x)y^*) - H^{op}_{\varepsilon'} H_\varepsilon
  (xy^*). \label{cotlar}
\end{eqnarray}
Note the  justified Cotlar's identity (\ref{cotlar}) holds for  all $\|\varepsilon_k\|\leq 1$ while in the commutative setting the    Cotlar's formula holds for
 $\varepsilon_k=\pm$ only.
\end{rem}

\begin{prop}\label{diag}
For any $x \in \m W,$ and any $p\geq 1$, and $\varepsilon_k\in {\m Z}( \B), \|\varepsilon_k\|\leq 1$
\begin{eqnarray*}\max\{ \|  E(H_\varepsilon x(H_\varepsilon x)^*)\|_p,
\|E(H_\varepsilon  (xH_\varepsilon ^{op}(x^*)))\|_p, \\ \|E(H_\varepsilon
  ^{op}(H_\varepsilon (x)x^*))\|_p,\|E(H^{op}_\varepsilon H_\varepsilon
  (xx^*))\|_p\}&\leq& \|E(xx^*)\|_p. \end{eqnarray*}
\end{prop}
\begin{proof}
Write $g=\sum_{\un i} g_{\un i}$ with $g_{\un i}\in \m W_{\un
  i}$. Then, by orthogonality of the $\m W_{\un i}$ over $\B$, all the
4 elements on the left hand side are of the form $\sum_{\un i} y_{\un
  i}E( g_{\un i}g_{\un i}^*) z_{\un i}^*$ with $y_{\un i},z_{\un
  i}\in \{1, \varepsilon_{i_n}\}$. But $\sum_{\un i}
y_{\un i}E( g_{\un i}g_{\un i}^*) z_i^*=\sum_{\un i} a_{\un i} y_{\un
  i}z_{\un i}^* a_{\un i}$ with $a_{\un i}=E( g_{\un i}g_{\un
  i}^*)^{1/2}$ so that the inequality follows by the H\"older
inequality as $\sum_{\un i} a_{\un i}^2=E(xx^*)$.
\end{proof}

 We can prove the main result

 \begin{thm}
\label{main}
 For  $1<p<\infty$, there is a constant $c_p$ so that for ${\varepsilon_k\in {\m Z}( \B)}, \|\varepsilon_k\|\leq 1$ and $x\in \m W$
\begin{eqnarray}\label{key1}  \|H_\varepsilon  x\|_{p}\leq { c_{p}}  \| x\|_{p}, \qquad \|H_\varepsilon^{op}  x\|_{p}\leq {c_{p}}  \| x\|_{p} .\end{eqnarray}

Moreover the equivalence holds with constant $c_p$ in both directions if $\varepsilon_k$'s are further assumed to be
unitaries.
\end{thm}
\begin{proof} Assume $\|H_\varepsilon \|_{L^p(\A)\rightarrow L^p(\A)}\leq c_p$.
 We will show that $ \|H_\varepsilon x\|_{2p }\leq (c_p+\sqrt
 {2c^2_{p}+4}) \| x\|_{2p}$ for all $x\in \m W$, and similarly for $H^{op}_\varepsilon$ using the $*$-operation. Once this is
 proved, we get the upper desired estimate  for all
 $p=2^n, n\in\N$, by induction and the fact that $\|H_\varepsilon
 x\|_{2 }= \|H_\varepsilon^{op}
 x\|_{2 }=\|x\|_{2} $. Applying interpolation and
 duality, we then get the result for all $1<p<\infty$ 
(note that the adjoint of $H_\varepsilon$ is $H_{\varepsilon^*}^{op}$). The equivalence holds for unitary $\varepsilon$ since
 $H_\varepsilon H_{\varepsilon^*}=id_\A$ in this case.  In fact, Cotlar's formula \eqref{cotlar}
 implies that for $x,\, y\in \m W$
\begin{eqnarray}
\bmr{H_\varepsilon x(H_{\varepsilon'} y)^*}  &=& \bmr{H_\varepsilon
  (xH_{\varepsilon'} ^{op}(y^*))}+\bmr{H_{\varepsilon'}  ^{op}(H_\varepsilon (x)y^*)}-\bmr{H^{op}_{\varepsilon'} H_\varepsilon
  (xy^*)}.\label{<1}
\end{eqnarray}

Apply H\"older's inequality and Proposition \ref{diag} to this identity for $x=y, \eps=\eps'$, we get
\begin{eqnarray*}
\| H_\varepsilon x\|_{{2p}}^2&\leq&2 c_{p}\|x \|_{2p}\| H_\varepsilon x\|_{2p}+ (4+c^2_{p})\|x \|^2_{2p}.
\end{eqnarray*}
That is $\|H_\varepsilon x\|_{{2p}}\leq (c_p+\sqrt  {2c^2_{p}+4})\| x\|_{{2p}}$.
 \end{proof}

\begin{rem}
As $\prod_{n=0}^\infty \frac {1+\sqrt{2+4/c_{2^n}^2}}{1+\sqrt{2}}<\infty$, one gets that for $p\geq 2$, $c_p\leq C p^\gamma$ with $\gamma= \frac {\ln (1+\sqrt2)}{\ln 2}$.
\end{rem}

 
\begin{rem}
By the usual trick to replace $\B,\, \A_k$ by $\B\otimes M_n$ and
$\A_k\otimes M_n$, one get that the maps $H_\varepsilon$ are completely bounded 
on $L^p$ for $1<p<\infty$.
\end{rem}

\begin{rem}\label{genM} We can use a slighter general definition for $H_\varepsilon$ by taking
$\varepsilon_k\in \B\otimes \m M$ where $\m M$ is a finite von Neumann
algebra, then $E(x)$ and $L_k(x)$ have to be understood as
$E(x)\otimes 1$ and $L_k(x)\otimes 1$. Theorem \ref{main}
 remains valid  with the assumption that $\varepsilon\in \m Z(\B)\otimes \m M$.  
\end{rem}


\subsection{Corollaries}

 In this section, we derive a few direct consequences of Theorem \ref{main}.

For any $k_0\in \N$, let $\varepsilon_{k_0}=-1$ and $\varepsilon_k=1$ for $k\neq k_0$. Then $  L_{k_0}=\frac {id_\A-H_\varepsilon}2$.
\begin{corollary}
\label{cor1}
For any $1<p<\infty$,  
 \[  \| L_{k}x\|_{p}\leq \frac {1+c_p}2 \| x\|_{p}.\]
\end{corollary}
 
\begin{corollary}
\label{corgen}
 For  $1<p<\infty$,  we have
\begin{eqnarray} 
\| (L_kx)_{k=0}^\infty  \|_{L^p(\ell_2^{cr})}\simeq^  { \sqrt2c_{p}}  \| x\|_{p}, \label{corgenLk} \\
\|  (R_kx) _{k=0}^\infty\|_{L^p(\ell_2^{cr})}\simeq^  { \sqrt2c_{p}}  \| x\|_{p}.
\end{eqnarray}
\end{corollary}
\begin{proof}
By duality we may only consider $1<p<2$. For any $x\in L^p$
$$\frac 1{c_p} E_\eps\| H_\eps(x)\|_p  \leq \|x\|_p\leq c_p E_\eps\| H_\eps(x)\|_p=c_p  E_\eps\|\sum_k \varepsilon_k L_kx\|_p.$$
We conclude by the non commutative Khintchine inequality (\ref{khin}) for $\varepsilon_k=z^{2^k}$.
\end{proof}

\begin{rem} We will prove a variant of   Corollary \ref{corgen} in the next subsection as Theorem \ref{improved}. 
\end{rem}
\begin{corollary}
  \label{JPXf}
For any $1<p<\infty$, any sequences  $(i_k)\in \N^{\N}$ and $(x_k)\in  L^p(\ell_2^c)$, we have 
 \begin{eqnarray}
  \|(\sum_{k=1}^{\infty }|L_{i_k}x_k|^2)^{\frac12}\|_{p}  \leq {c_{p}} \| (\sum_{k=1}^\infty|x_k|^2)^\frac12\|_{{p}} \label{klcf}\\
  \|(\sum_{k=1}^{\infty }|R_{i_k}x_k|^2)^{\frac12}\|_{{p}}  \leq c_{p} \| (\sum_{k=1}^\infty|x_k|^2)^\frac12\|_{{p}}. \label{krcf}
  \end{eqnarray}
\end{corollary}
\begin{proof}
Fix  a sequence $\varepsilon_k=\pm 1$ and apply Theorem $\ref{main}$   to
$x=\sum_l \varepsilon_{i_l} x_l \otimes c_l\in L^p(\A\otimes B(\ell_2))$. We have
$$\| \sum_{k,l} \varepsilon_k\varepsilon_{i_l} L_k(x_l) \otimes c_l\|_p \leq
c_p\|\sum_l \varepsilon_{i_l} x_l \otimes c_l\|_p=c_p\| (\sum_{l=1}^\infty|x_l|^2)^{\frac12}\|_p.$$
 Let $\varepsilon_k$ to be Rademacher variables, we have 
 $$\|(\sum_{l=1}^{\infty }|L_{i_l}x_l|^2)^{\frac12}\|_{p} =\|E \sum_{k,l} \varepsilon_k\varepsilon_{i_l} L_k(x_l) \otimes c_l \|_{p}
 \leq c_p\| (\sum_{l=1}^\infty|x_l|^2)^{\frac12}\|_p. $$

The proof of the second inequality is similar.
\end{proof}

\begin{rem} Lemma \ref{JPXf} was proved in \cite{JPX07} (Lemma 2.5, Corollary 2.9) for  $x_k$'s supported on reduced words with length $=d$ with constants depending on $d$, independent on $p$.  
\end{rem}




\subsection{ Length independent estimates for Rosenthal's inequality }
We will apply Theorem \ref{main} to obtain a length free estimate for
the Rosenthal's inequality proved in \cite{JPX07} (Theorem A).  In this
subsection, we restrict $\eps\in \{\pm1\}^\N$ and $\eps_0=0$ in the
form of $H_\eps= \sum_{k\in \N} \eps_k L_{k}$ and $H^{op}_\eps$.  When
no confusion can occur, we use the notation $T$ instead of $T \otimes
Id$ for its ampliation.

Thanks to the previous results, we can define the following paraproduct for
$x\in L^p(\A)\otimes L^p(\M)$ ($1<p<\infty$)
and $y\in L^q(\A)\otimes L^q(\M)$
with $\frac 1 p+\frac  1 q>1$
$$ x\ddag y= E_\eps H_\eps(H_\eps(x) y)=\sum_{k\in \N} L_k((L_kx) y),$$
with $E_\eps$ the expectation with respect to the Haar measure on $\{\pm1\}^\N.$
We also set
$$x\dag y= xy- x\ddag y - E(xy)=\sum_{k=0}^\infty \mathring{L^\bot_k} ((L_kx) y).$$
Here $\mathring{L^\bot_k} =\sum_{j\neq k, j\in \N} L_j $  for any $k\geq 0$.

If $x$ and $y$ are elementary tensors ($x\notin \B$), $x\ddag y$ collects in the reduced form of $xy$ all elements whose first letter is in the same algebra as $x$ while $x\dag y$ collects the rest in the reduced form of $xy$. 

\begin{prop}\label{paraid}
  We have the following, for $1<p<\infty$, $1<q\leq \infty$ with $\frac 1 p + \frac 1 q=\frac 1 r>1$
  \begin{enumerate}[i)]
  \item $\|H_\eps(x\ddag y)\|_r\leq c_rc_p \|x\|_p \|y\|_q$, $\|x\dag y\|_r\leq (2+c_rc_p) \|x\|_p \|y\|_q.$
  \item $H_\eps (x\ddag y)=H_\eps(x) \ddag y$, $x\dag H_\eps^{op}(y)=H_\eps^{op}(x\dag y).$

    In particular $x\ddag y \in \mathcal L_k$ if $x\in \mathcal L_k$
    and $x\dag y\in \mathcal R_k$ if $y\in \mathcal R_k$.
    \end{enumerate}
\end{prop}

\begin{proof}
  (i) simply follows from Theorem \ref{main} and the definitions. We now prove (ii). For $\ddag$, this follows
  from its definition. 

  For $\dag$, we check the following formula from which the identity because of the translation invariance of the Haar measure on $\{-1,1\}^\N$:
  \begin{equation}
  x\dag y=E_{\eps'} \Big(  {H_{\varepsilon'}^{op}(x\dag H_{\eps'}^{op}(y))\Big)}.\label{idsharp}
  \end{equation}
  We first notice that the identity holds if $x\in \B$
  as $x\ddag y=0$ and $x\dag y=x(y-E(y))$. Similarly if $y\in \B$,
  $x\ddag y=(x-E(x))y$ and $x\dag y= 0$.
  Thus we can assume $E(x)=E(y)=0$.  Apply the Cotlar identity (\ref{cotlar})   to $ H_\eps(x)$ and
  $H_{\eps'}^{op}(y^*)$ and note $H_\eps^2(x)=x$ and
  $H_{\eps'}^{op 2}(y^*)=y^*$, we get $${xy}-E xy={H_\varepsilon(H_\eps(x)y)}+{H_{\varepsilon'}^{op}(xH_{\eps'}^{op}(y))}-{H^{op}_{\varepsilon'} H_\varepsilon(H_\eps(x)H_{\eps'}^{op}(y))}.$$
Taking expectations with respect to $\eps$ and $\eps'$ gives (\ref{idsharp}). One can also verify directly the identity for $\dag$ in (ii) by its bilinearity and looking at 
elementary tensors  $x,y \in \m W$ and using Proposition \ref{ghid} (iv)-(v).
\end{proof}

 \begin{rem}\label{imping}
There are situations for which one can slightly improve those
inequalities. For instance if $r=2$, then $\|x\dag y\|_r\leq (1+c_p)
\|x\|_p \|y\|_q$. Or in general $\|x\ddag y\|_r\leq c_r \sup_\eps
\|H_\eps(x)\|_p \|y\|_q$ and $\|x\dag y\|_r\leq (2+c_r) \sup_\eps
\|H_\eps(x)\|_p \|y\|_q.$
\end{rem}

\begin{lemma} \label{inter}
For $2\leq p<\infty $ and $x\in L^p(\A)$
\begin{eqnarray}
\|\sum_{k\in \N}\bmr{L_k(x)L_k(x)^*} \|_{\frac p2}&\leq& \alpha_p   \|\sum_k |(L_kx)^*|^2\|_{\frac p2}^\frac12(\sum_k \|L_kx\|^p_p)^\frac1p \label{newkey}\\ 
\|\sum_{k\in \N}\bmr{R_k(x)^*R_k(x)}\|_{\frac p2}&\leq& \alpha_p \|\sum_{k\in \N} |R_k(x)|^2 \|_{\frac p2}^\frac12 (\sum_k \|R_kx\|^p_p)^\frac1p
  \end{eqnarray}
  with $\alpha_p\leq 3c_4^2$ for $2<p\leq 4$ and $\alpha_p\leq 2\sqrt2 (c_{\frac p2}^2+c_{\frac p2})$ for $p\geq 4$.
\end{lemma}

\begin{proof}   
 Let us assume $p\geq4$ first. 
We use the decomposition $L_k(x)L_k(x)^*-E(L_k(x)L_k(x)^*)= L_kx \ddag (L_kx)^*+ L_kx \dag (L_kx)^*$. By Corollary \ref{corgen} and
 Proposition \ref{paraid}, as $L_k(x)\ddag   L_k(x)^*\in \mathcal L_k$ we have
  \begin{eqnarray*}
  && \| \sum_{k\in \N} L_k(x)\ddag   L_k(x)^*\|_{\frac p2}\\
  &\leq&  \sqrt2 c_{\frac p2}\max \Big[
  \| \sum_{k\in \N} L_k(x)\ddag   L_k(x)^* \tens c_k \|_{\frac p2},
  \| \sum_{k\in \N} L_k(x)\ddag   L_k(x)^* \tens r_{k} \|_{\frac p2}\Big].
   \end{eqnarray*} 
   Using the bilinearity of $\ddag$, we have
  \begin{eqnarray*}
    \sum_{k\in \N} L_k(x)\ddag L_k(x)^*\tens r_k & =& \sum_{k\in \N}
    (L_k(x)\tens r_k)\ddag (L_k(x)^*\tens d_k)\\
     & =& E_\eps (\sum_k \eps_k L_k(x)\tens r_k)\ddag (\sum_k \eps_k L_k(x)^*\tens d_k)\nonumber\\
    & =& E_\eps\Big[ H_\eps(\sum_k L_k(x)\tens r_k)\ddag H_{ \eps}^{op}(\sum_k L_k(x)^*\tens d_k)\Big] 
  \end{eqnarray*}
  So   we can conclude from  Theorem \ref{main} and Remark \ref{imping}  that
  \begin{eqnarray}
  \|\sum_{k\in \N}L_k(x)\ddag L_k(x)^*\otimes r_k\|_{\frac p2}&\leq& c_{\frac p2} 
 \sup_{\eps,\eps'} \|H_{\eps'} \sum_k L_kx\otimes r_k\|_p \|H_\eps^{op}\sum_k L_k(x)^*\tens d_k\|_p\nonumber\\
 &\leq&c_{\frac p2}    \|\sum_k L_k(x)\tens r_k\|_p(\sum_k \|L_kx\|^p_p)^\frac1p.\label{rg} 
      \end{eqnarray}
 Similarly we have 
 \begin{eqnarray}
 \|\sum_{k\in \N}L_k(x)\ddag L_k(x)^*\otimes c_k\|_{\frac p2}&=&E_\eps\Big[ H_\eps(\sum_k L_k(x)\tens d_k)\ddag H_{ \eps}^{op}(\sum_k L_k(x)^*\tens c_k)\Big] \nonumber\\
 &\leq& c_{\frac p2}  (\sum_k \|L_kx\|^p_p)^\frac1p \|\sum_k L_k(x)\tens r_k\|_p. \label{xy}
 \end{eqnarray}
 Combining these two estimates we get 
 $$  \| \sum_{k\in \N} L_k(x)\ddag   L_k(x)^*\|_{\frac p2}\leq \sqrt 2c^2_{\frac p2}   \|\sum_k L_k(x)\tens r_k\|_p(\sum_k \|L_kx\|^p_p)^\frac1p,$$   for $p\geq 4$.
  We can treat the $\dag$ term similarly since $L_kx\dag (L_kx)^*\in {\cal R}_k$ and get 
  $$  \| \sum_{k\in \N} L_k(x)\dag   L_k(x)^*\|_{\frac p2}\leq \sqrt 2c_{\frac p2}(2+c_{\frac p2})   \|\sum_k L_k(x)\tens r_k\|_p(\sum_k \|L_kx\|^p_p)^\frac1p.$$ 
  We then get (\ref{newkey}) for $p\geq 4$ with constant $2\sqrt2 (c_{\frac p2}^2+c_{\frac p2})$.

To deal with the remaining cases, we will use interpolation by proving
a better bilinear inequality for $2\leq p\leq 4$.
$$\|\sum_{k\in \N}\bmr{L_k(x)R_k(y)} \|_{\frac p2}\leq  3 c_4^2  \|\sum_k L_kx \tens r_k\|_{p}(\sum_k \|R_ky\|^p_p)^\frac1p.$$
The spaces consisting of elements of the form  $\sum_k L_kx \tens r_k$ and $\sum_k R_ky \tens d_k$ are 
$c_p$-complemented in $L^p(\A)\tens S_p$ by Theorem \ref{main}. Hence the norms on the right hand side interpolate for $2\leq p\leq 4$ (both 
with constant $c_4^{2(1-2/p)}$).

We just need to justify the endpoint inequalities.
For $p=2$, we have by H\"older's inequality  that \begin{eqnarray*}
\|\sum_{k\in\N}\bmr{L_k(x)R_k(y)}\|_{1}
&\leq&2\|\sum_k L_k(x)\tens r_k\|_2\|\sum_k R_k(y)\tens c_k\|_2\\
&=&2\|\sum_k L_k(x)\tens r_k\|_2(\sum_k \|L_ky\|^2_2)^\frac12.
\end{eqnarray*}
For $p=4$, by orthogonality and as in \eqref{rg} 
\begin{eqnarray*}
   \| \sum_{k\in \N} L_k(x)\ddag   R_k(y)\|_{2}
  &=& \| \sum_{k\in \N} L_k(x)\ddag   R_k(y)\tens r_k\|_{2}\\
&\leq &  \|\sum_k L_kx \tens r_k\|_{4}(\sum_k \|R_ky\|^4_4)^\frac14.
   \end{eqnarray*}
Similarly we get $\| \sum_{k\in \N} L_k(x)\dag   R_k(y)\|_{2}\leq 2
 \|\sum_k L_kx \tens r_k\|_{4}(\sum_k \|R_ky\|^4_4)^\frac14.$
Thus by interpolation we get (\ref{newkey}) for $2<p<4$ with a constant $3c_4^{4(1-2/p)}$. 
\end{proof}
 
  \begin{thm}\label{improved}
  For $2\leq p<\infty$ and $x\in L^p$ 
$$ \beta_p^{-1}\|x \|_p\leq  \max\Big\{\|(\sum_{k=0}^\infty |L_k(x)|^2)^\frac12 \|_p,\| E (xx^* )\|_{\frac p2}^{\frac1 2}\Big\}\leq \sqrt 2 c_p\|x\|_p $$
  and  $$ \beta_p^{-1}\|x \|_p\leq  \max\Big\{\|(\sum_{j=0}^\infty |R_j(x)^*|^2)^\frac12 \|_p,\| E (x^*x )\|_{\frac p2}^{\frac 12}\Big\}\leq \sqrt 2 c_p \|x\|_p $$
  with $\beta_p\leq\sqrt2c_p(1+\alpha_p)\lesssim c_p^3$.
  \end{thm} 
   
\begin{proof}
For the first equivalence, the upper inequality follows from Corollary
\ref{corgen}. For the lower bound, by Lemma \ref{inter}
as $E(xx^*)=\sum_{k\geq 0} E (L_k(x)L_k(x)^*)$ and $\bmr{L_0(x)L_0(x)^*}=0$,
$$\| \sum_{k\geq 0} L_k(x)\tens r_k\|_p \leq   \alpha_p 
\| \sum_{k\in \N} L_k(x)\tens d_k\|_p + \| E (xx^* )\|_{p/2}^{1/2}.$$
But as $p\geq 2$ the map $c_k\mapsto d_k$ is a contraction on $L^p$, so we deduce
$$\| \sum_{k\geq 0} L_k(x)\tens r_k\|_p\leq  \alpha_p \| \sum_{k\geq 0} L_k(x)\tens c_k\|_p + \| E (xx^* )\|_{p/2}^{1/2},$$
and we conclude the lower bound by  Corollary
\ref{corgen} again.   The other inequality follows by taking adjoints.\end {proof}
  We get the following Rosenthal type inequality as a direct application.
\begin{corollary}\label{ros}
  Let $2<p<\infty$ 
  \begin{enumerate}[(i)]
  \item For $x=\sum_{k=0}^\infty a_k$ with $a_k\in \mathcal L_k$, we have
   $$ { \beta_p^{-2}} \|x\|_p \leq \|E( xx^*)\|_{\frac p2}^{\frac 12}+ \|E(x^*x) \|_{\frac p2}^{\frac 12}+ \|\sum_{k,j}R_j(a_k)\otimes e_{k,j}\|_p \leq { (2c_p^2+2)} \|x\|_p$$
  \item For $x=\sum_{k=0}^\infty a_k$ with $a_k\in \mathcal L_k\cap \mathcal R_k$, we have 
  $${ \beta_p^{-2}} \| x\|_p  \leq \|E(xx^*)\|_{\frac p2}^{\frac12}+ \|E(x^*x)\|_{\frac p2}^{\frac 12}+ (\sum_k \|a_k\|_p^p)^{\frac 1p} \leq { (2c_p^{2}+2)}\|x\|_p. $$
\end{enumerate}
  \end{corollary}
 
\begin{proof}
Apply Theorem \ref{improved} twice and notice that $(r_k\tens c_k)$ generate the canonical basis of $\ell_p$ in $S_p$. We get (i) and (ii) follows immediately. 
 \end{proof} 
 
\begin{rem}
We point out that   Corollary \ref{ros} (ii)   was  proved in \cite{JPX07} (Theorem A)
  when $a_k$ are supported on reduced words with a fixed length  with constants  
independent of $p$ but   depends  on the length.  
Noticing that by the Khintchine inequalities from
\cite{RX06}, $H_\eps$ and $H_\eps^{op}$ are bounded on words of length
at most $d$ with a constant that depends only on $d$ and $\A$, thus by
interpolation, the  argument above also implies Corollary \ref{ros}  (ii) with constants independent of $p$ but dependent on the length. 
\end{rem}


\begin{rem}
All the results of this section also hold in the completely bounded setting.
\end{rem}

\section{Free groups}
We can apply the previous results to the free group as it is naturally
a free product. Let $g_i$, $i\in \N$ be the set of generators of
$\F_\infty$. We let $\m L_{g_i}$ and $\m L_{g_i^{-1}}$ be the set of
reduced word starting by $g_i$ and $g_i^{-1}$ respectively and $\m
L_{g_i^{\pm}}=\m L_{g_i}\cup\m L_{g_i^{-1}}$.  We denote by $
L_{g_i}$, $L_{g_i^{-1}}$ and $L_{g_i^{\pm}}$ the associated
projections. We use the notation $R_{g_i}$ and $\mathcal R_{g_i},...$ for the right analogues. We will often use the convention $g_i=g_{-i}^{-1}$ for
$i<0$ so that $L_{g_i^{-1}}=L_{g_{-i}}$ for any $i\in \Z^*$. Finally $S$ will denote
the set $\{g_i\;;\;i\in \Z^*\}$.

Let $\M$ be a finite von Neumann algebra. Theorem \ref{main} immediately gives that, for any $x\in L_p(\hat\F_\infty)\otimes L_p({\M}), 1<p<\infty$ and  
 sequences of unitaries $\varepsilon_i\in \m Z(\m M), \|\varepsilon_k\|\leq 1$,
\begin{equation}\label{freegen}\| (Id\otimes\tau) x + \sum_i \varepsilon_i (Id\otimes
 L_{g_i^{\pm}})(x)\|_p\simeq^{ c_{p}} \| x\|_{p}.\end{equation}

 We slightly extend it
\begin{thm}
\label{main2}
Let $(\varepsilon_k)_{k\in \Z}$ be a sequence in   $\m Z(\m M), \|\varepsilon_k\|\leq 1$. Then for any $x\in L^p(\hat \F_\infty)\otimes L^p(\m M)$ and $1<p<\infty$
$$\|\varepsilon_0 (Id\otimes\tau) x + \sum_{k\in \Z^*}^\infty \varepsilon_k 
(Id\otimes
 L_{g_k})(x)\|_p\leq c_{p} \| x\|_{p}.$$ 
 The equivalence holds if we assume further that $\varepsilon_k$ are unitaries in $\m Z(\m M)$.
 \end{thm}
\begin{proof}
 We may assume $\varepsilon_0=1$. We consider the following group
 embedding $\pi :\F_\infty \to \F_\infty*\F_\infty$ defined by
 $\pi(g_i)=g_ih_i$ where $(h_i)$ are the free generators of the second
 copy of $\F_\infty$.  This extends to a complete isometry for
 $L^p$-spaces and one checks directly that $$( \sum_{k=0}^\infty
 \varepsilon_{-k}L_{h_k^{\pm}}+ \tau + \sum_{k=0}^\infty
 \varepsilon_{k}L_{g_k^{\pm}})\circ \pi = \pi\circ ( \sum_{k=0}^\infty
 \varepsilon_{-k}L_{g_k^{-1}}+ \tau + \sum_{k=0}^\infty
 \varepsilon_{k}L_{g_k}).$$
The statement follows  from the amalgamated version of \eqref{freegen}.
 \end{proof}

The proof of Lemma \ref{inter} and Theorem \ref{improved} can easily be adapted to the free group
where $H_\eps=\eps_e L_e +\sum_{h\in S} \eps_h L_h$ with $|\eps_h|=1$ and the convention $L_e x=\tau x$. We simply give the result
 \begin{thm}\label{improvedfree}
  For $2<p<\infty$, $x\in L_p(\hat\F_\infty)\otimes L_p(\M)$
  $$\beta_p^{-1}\|x \|_p\leq \max\Big\{\|(\sum_{|h|\leq 1} |L_h(x)|^2)^\frac12 \|_p, \| (\tau\tens Id) (xx^*))\|_{p/2}^{1/2}\Big\}\lesssim {\sqrt 2 c_p}\|x\|_p $$
  and  $$\beta_p^{-1}  \|x \|_p\leq \max\Big\{\|(\sum_{|h|\leq1}|R_h(x)^*|^2)^\frac12 \|_p,\| (\tau\tens Id) (x^*x )\|_{p/2}^{1/2}\Big\}\leq {\sqrt 2c_p}\|x\|_p. $$
  \end{thm} 
  \begin{corollary}\label{rosfree}
  Let $2<p<\infty$ 
  \begin{enumerate}[(i)]
  \item For $x=\sum_{k=-\infty}^\infty a_k$ with $a_k\in \mathcal L_k$, we have
   $$ { \beta_p^{-2}} \|x\|_p \leq \|(\tau\tens Id)( xx^*)\|_{\frac p2}^{\frac 12}+ \|(\tau\tens Id)(x^*x) \|_{\frac p2}^{\frac 12}+ \|\sum_{k,j}R_j(a_k)\otimes e_{k,j}\|_p \leq { (2c_p^2+2)} \|x\|_p$$
  \item For $x=\sum_{k=-\infty}^\infty a_k$ with $a_k\in \mathcal L_k\cap \mathcal R_{\phi(k)}$ and $\phi:\Z\mapsto \Z$ an one to one map, we have 
  $${ \beta_p^{-2}} \| x\|_p  \leq \|(\tau\tens Id)(xx^*)\|_{\frac p2}^{\frac12}+ \|(\tau\tens Id)(x^*x)\|_{\frac p2}^{\frac 12}+ (\sum_k \|a_k\|_p^p)^{\frac 1p} \leq { (2c_p^{2}+2)}\|x\|_p. $$
\end{enumerate}
  \end{corollary}
 
\begin{proof}
Apply Theorem \ref{improvedfree} twice we get (i).  (ii) follows immediately because $\phi$ is one to one. 
 \end{proof}

\begin{rem}
All results before this subsection hold for free groups with $L_k,
R_k$ replaced by $L_{g_k}$ (resp. $L_{g^{-1}_k}$ or $L_{g^\pm_k}$) and
$R_{g_k}$ (resp. $R_{g^{-1}_k}$ or $R_{g^\pm_k}$ ). We can strengthen
some of them. These will be recorded in the following.
\end{rem}

Given $g,h$ reduced word of $\F_\infty$, we write $g\leq h$ (or $h\geq
g$) if $h=gk$ with $g,h,k$ reduced words, i.e. $|g^{-1}h|=|h|-|g|$. We
write $g\nleqslant h$ otherwise. Let 
\begin{eqnarray*}
\mathcal L_h:=\{g\in \F_\infty, g\geq
h\}
\end{eqnarray*}
 and $L_h$ the associated $L^2$-projection, this is compatible with our previous notation.

\begin{corollary}
\label{Lh}
 For any $1<p<\infty$, $h\in \F_\infty$ and $x\in L^p(\hat \F_\infty)\otimes L^p(\m M)$ ,  
 \[  \| L_{h}x\|_p \leq{\frac{ c_{p}+1}2} \| x\|_{L^{p}}.\]
 Moreover, $\lim_{|h|\to \infty} \|L_hx\|_p\to 0$.  
\end{corollary}
\begin{proof}
 Without loss of generality, we may assume $h\in \mathcal R_{g_1}$ and
 $h=h'g_1$.  Then $L_h x= \la_{h'}L_{g_1} (\la_{h'^{-1}}x)$. The $L^p$
 bounds follows from Theorem \ref{main2}.  Note the $L^p$ space is
 defined as the closure of $C_c(\F_\infty)$, we get the convergence by
 the uniformly boundedness of $L_h$ on $L^p$.
 \end{proof}

\begin{corollary}
  \label{J.P}
For any $1<p<\infty$, any sequences  $(h_k)\in \F_\infty\setminus\{e\}$  and $(x_k)\in  L^p(\ell_2^c)$, we have 
 \begin{eqnarray}
  \|(\sum_{k=1}^{\infty }|L_{h_k}x_k|^2)^{\frac12}\|_{p}  \leq {c_{p}} \| (\sum_{k=1}^\infty|x_k|^2)^\frac12\|_{{p}}. \label{jpklcf}
  \end{eqnarray}
  \end{corollary}

\begin{proof} Let us assume such that $h_{k}\in {\cal R}_{g_{i_k} }, i_k\in \Z$.  Assume 
  $h_k=h_k'g_{i_k} $. 
Then $$L_{h_k} x_k= \la_{h_k'}L_{g_{i_k} } (\la_{h_k'^{-1}}x_k).$$ 
So $$\sum_{k=1}^{\infty }|L_{h_k}x_k|^2=\sum_{k=1}^{\infty }|L_{g_{i_k} } (\la_{h_k'^{-1}}x_k)|^2.$$
We get the result  by the free group version of Corollary \ref{JPXf}. 
\end{proof}
\bigskip

\subsection{A length reduction formula}
Let $\m W_{\geq d}$ be the set of word in $\F_\infty$ of length
greater than $d$, also denote by $W_{\geq d}$ the subspace in $L^p$
generated by $\lambda_w,\,w\in\m W_{\geq d}$.  For $w\in \F_\infty$,
we let $w_l$ denote its $l$-th letter (if it exists) and $\partial
w=w_1^{-1}w$.

Take any $x=\sum_{w\in \m W_{\geq 1}} x_w \lambda_w \in W_{\geq 1}$, we have 
$$\|(\sum_{h \in S} |L_h(x)|^2)^{\frac 12}\|_p=\| \sum_{w\in \m W_{\geq 1}} x_w \lambda_w \tens c_{w_1}\|_p=\| \sum_{w\in \m W_{\geq 1}} x_w \lambda_{\partial w} \tens c_{w_1}\|_p$$

At the operator space level, Theorem \ref{improvedfree} means that the
map $\iota: W_{\geq d} \to C_p\tens W_{\geq d-1}\oplus R_p$ given by
$\iota(\lambda_w)=\lambda_{\partial w}\otimes c_{w_1} \oplus r_w$ is a
complete isomorphism. Iterating, we obtain a complete isomorphism for $2<p<\infty$
$$\iota_d : \left\{\begin{array}{lcl} W_{\geq d} &\to& C_p^{\tens
  d}\tens L_p(\hat \F_\infty) \oplus C_p^{\tens d-1}\tens R_p \oplus
...\oplus C_p\tens R_p \oplus R_p \\ \lambda_{w} & \mapsto&
c_{w_1,...,w_d} \tens \lambda_{\partial^dw} \oplus c_{w_1,...,w_{d-1}}
\tens r_{\partial^{d-1}w}\oplus ... \oplus c_{w_1}\tens r_{\partial w}
\oplus r_{w}
\end{array}\right.$$
Let us state this as a Corollary, which   generalizes  the result of \cite{PP05}. 
\begin{corollary}(Length reduction formula) \label{PP} For any $d\geq 1$, $\iota_d$ extends to a completely isomorphism such that for $x\in   W_{\geq d}$, $2\leq p<\infty$, $$\beta_p^{-d}\|x\|_p\leq \|\iota_d x\|\leq (\sqrt 2 c_p)^d \|x\|_p,$$
for all $x\in L^p(\hat\F_\infty)$.
\end{corollary}

Fix some $d\in \N$
and  any reduced word $w=w_1...w_n$ in the generators,  we define: 
$$L_h^{(d)}(\la_w)= \delta_{w_k=h}\la_w,\qquad   \textrm{and} \qquad
H_\eps^{(d)}=\eps_e P_{d-1}+\sum _{h\in S} \eps_hL_h^{(d)},$$
for any choice of  $\eps_h, |h|\leq 1$ with $|\eps_h|\leq 1$.
Recall that by  \cite{RX06} or \cite{JPX07}, $P_{d-1}$ is completely bounded  on $L^p$  (this also follows from Theorem \ref{improvedfree}). Note that 
$$\|\iota_{d-1} H_{\eps(1)}^{(1)}H_{\eps(2)}^{(2)}\cdots H_{\eps(d)}^{(d)}x\|=\|H_{\eps(d)} \iota H_{\eps(d-1)} \iota\cdots \iota H_{\eps(1)}x\|.$$  We get immediately
\begin{thm}\label{hn}
For any $d\geq 1$ and $x\in L^p(\hat\F_\infty)\otimes L^p(\m \M), 1<p<\infty$, $$\|H_{\eps(1)}^{(1)}H_{\eps(2)}^{(2)}\cdots H_{\eps(d)}^{(d)} x\|_p\simeq^{c_{p,d}} \|x\|_p$$ 
with $c_{p,d}\leq {(\sqrt 2c_p) ^{2d-1}\beta_p^{d-1}}\lesssim c_p^{5d-4}$ and $\|H_{\eps}^{(d)} x\|_p\simeq^{(\sqrt 2c_p) ^d\beta_p^{d-1}} \|x\|_p$ for any choice of $|\eps_k|= 1$. 
\end{thm}


\medskip

We give a faster argument for the boundedness of $H_\eps^{(d)}$.
 Consider, $\varepsilon_h=\pm1$ for $h\in \F_\infty$,  let
$$ H_\eps^{Ld} =\eps_eP_{d-1}+\sum _{h\in \F_\infty, |h|= d} \varepsilon_h L_h,\ \ \ H_{\eps}^{Rd}
=\eps_eP_{d-1}+\sum _{h\in \F_\infty, |h|= d} \varepsilon_{h^{-1}} R_h.$$
Recall that $L_h$ (resp. $R_h$) is defined as the projection onto the set of all reduced words starting (resp. ending) with $h$.
We get $H_\eps^{(d)}$ from $H_\eps^{Ld}$ if $\eps_h$ depends only on the $d$-th letter of $h$.
\begin{corollary}
  \label{d}
For any $1< p<\infty$,   we have for any $x\in L^p(\hat\F_\infty) $,
 \begin{eqnarray}
  \| x\|_p\simeq  \|H_\eps^{Ld} x\|_p\simeq \|(L_h x )_{|h|=d}\|_{L^p(\hat\F_\infty,\ell_2^{cr} )}
  \end{eqnarray}
\end{corollary}
\begin{proof}
 Note that  a similar identity to (\ref{cotlar}) holds for free groups with $H_\eps^{Ld}$ and any $g,h$ with $|g^{-1}h|\geq {2d-1}$. We then have that
 \begin{eqnarray}\label{haag1p}
 &&P^\bot_{  2d-2}[H_\eps^{Ld} x(H_\eps^{Ld} x)^*]\nonumber\\&=& P_{ 2d-2}^\bot[H_\eps^{Ld}
  (xH_\eps^{Ld} (x^*))+H_\eps^{Ld}   (H_\eps^{Ld} (x)x^*)-H^{Rd }_\eps H_\eps^{Ld}
  (xx^*)].
 \end{eqnarray}
 Let $c_{p,d}, p\geq 2$ be the best constant $c$ so that $\|H_\eps^{Ld}x\|_p\leq c\|x\|_p$. Recall that by the Haagerup inequality, the $L^1$ and $L^p$ norms are 
equivalent on the range of $P_{  2d-2}$:
 \begin{eqnarray*}
 \|P_{2d-2}[H_\eps^{Ld} x(H_\eps^{Ld} x)^*]\|_{p}\leq (2d-1)^{2-\frac 2p}\|H_\eps^{Ld} x(H_\eps^{Ld} x)^*\|_{1}\leq  (2d-1)^{2-\frac 2p}\| x\|^2_{2p},\\
 \|P_{2d-2}[H_\eps^{Ld}(xH_\eps^{Ld} (x^*))]\|_{p}\leq (2d-1)^{1-\frac 2p}\|H_\eps^{Ld}
  (xH_\eps^{Ld} (x^*))\|_{2}\leq  (2d-1)^{1-\frac 2p}c_{2p}\|  x\|^2_{2p}
 \end{eqnarray*}
for any $p>2$. Therefore,   $$c^2_{2p,d}\leq 2(2d-1)^{2-\frac2p}+2c_{2p,d}(2d-1)^{1-\frac2p}+2c_{p,d}c_{2p,d}+c_{p,d}^2.$$  
We then have 
$$c_{2p,d}\leq (2d-1)^{1-\frac2p}+c_{p,d}+\sqrt2( c_{p,d}+3(2d-1)^{1-\frac1p}).$$
Asymptotically  $c_{p,d}\simeq p^ \frac{\ln(1+\sqrt 2)}{\ln 2}$ for $d$ given and $c_{p,d}\simeq d^{1-\frac 2p}$ for $p$ given. So 
$$\|H_\eps^{Ld} x\|_{{p}}\leq c_{p,d} \|  x\|_{{p}}.$$
Since $H_\eps^{Ld}H_\eps^{Ld}=id$, we get the equivalence. The $1<p<2$ case follows by duality.
 \end{proof}

\begin{rem}
A straight forward c.b. version of Corollary \ref{d} is false for $\F_\infty$ (true for $\F_n$ with a constant depending on $n$ though). This is because the
operator valued Haagerup inequality is an equivalence between the
$L^p$ norm and the more complicated norm given by Corollary \ref{PP}. For instance it yields that the
set $\{\lambda(g_ig_j)\}$ is not completely unconditional, this would be a 
direct consequence of Corollary \ref{d}. 

\end{rem}

For any $x\in {\cal L}(\F_n), n<\infty$ and any choice of signs $H_\varepsilon
x$ can be viewed as an unbounded operator on $L^2(\hat \F_n)$ with
domain $C_c(\F_n)$. As usual $\mathbb K$ stands for the compact operators. Ozawa asked in \cite{O10} whether the commutator $[R_h,x]$ sends the unit ball of ${\cal L}(\F_n)$ into a compact set of $L^2(\hat\F_n)$ for any $h\in \F_n$ and $x\in {\cal L}(\F_n)$ and pointed out that the $L^p$-boundedness of $R_h$ implies a positive answer. We record a general result in the following corollary.
 
\begin{corollary}
  \label{Ozawa}
   We have for $d\in \N$ and  any choice of signs  $\varepsilon$
 \begin{enumerate}[(i)]
  \item $[H^{Rd}_\eps, x]\in B(L^2(\hat \F_n))$ 
  if  $x=x_1+H_{\varepsilon'}^{Ld} x_2$ for some $|\eps'|\leq 1, x_1,x_2\in {\cal L}(\F_n)$.
  \item $[H^{Rd}_\eps, x]\in {\mathbb K}(L^2(\hat \F_n))$ 
  for all  $x=x_1+H^{Ld}_{\varepsilon'} x_2$ for some   $|\eps'|\leq 1, x_1,x_2\in C_\la^*(\F_n)$. 
  \item $[H^{Rd}_\eps, x] $ maps the closed unit ball of ${\cal L}(\F_n)$ into a compact subset of $L^2(\hat\F_n)$ 
  if   $x\in L_{p}(\hat\F_n) $ for some $p>2$ (in particular, if $x\in {\cal L}(\F_n)$).
  \end{enumerate}
\end{corollary}

\begin{proof}
Similar to (\ref{haag1p}), we have
\begin{eqnarray*}
 P^\bot_{  2d-2}  [(H^{Ld}_{\varepsilon'} x)(H^{Rd}_\eps y)]= P^\bot_{  2d-2} [H^{Ld}_{\varepsilon'}
  (x(H_\eps^{Rd} y ))+H_\eps^{Rd} ((H^{Ld}_{\varepsilon'} x)y)-H^{Ld}_{\varepsilon'} H^{Rd }_\eps (xy)].  
  \end{eqnarray*}
  So, up to a finite rank perturbation, for $y\in L^2(\hat\F_n)$
  $$[H^{Rd}_\eps,H^{Ld}_{\varepsilon'} x]y=-H^{Ld}_{\varepsilon'}
  (x(H_\eps^{Rd} y ))+  H^{Ld}_{\varepsilon'} H^{Rd}_\eps (xy)=H^{Ld}_{\varepsilon'}([H^{Rd}_\eps,x]y).$$
  Therefore, for $x=x_1+H^{Ld}_{\varepsilon'} x_2$, up to a finite rank perturbation
  $$[H^{Rd}_\eps, x]=[H^{Rd}_\eps, x_1]+H^{Ld}_{\varepsilon'}([H^{Rd}_\eps,x_2]).
$$
  This implies (i). Note  that    $[R_h,\la_g]$ is finite rank for each $h,g$. We have that $[H^{Rd}_\eps, x]\in {\mathbb K}(\ell_2(\F_n))$ for all $x\in C_\la^*(\F_n)$. So  (ii) is true. For (iii), following the argument of Ozawa, we have, by H\"older's inequality and Theorem \ref{main2}
  $$\|[H^{Rd}_\eps, x]y\|_{L^2(\hat\F_n)}\lesssim \|x\|_{L^p(\hat\F_n)}\|y\|_{L^q(\hat\F_n)}$$
  for any $y\in  {L^q(\hat\F_n)} ,  \frac1q+\frac1p=\frac12 $.
   By  density of $C_\la^*(\F_n)$ in $L^p(\hat\F_n), p<\infty$ and since 
${\cal L}(\F_n)\subset L^p(\hat\F_n)$ contractively, we get the desired result.
  \end{proof}
\begin{rem} When $n=1$, the  space of functions $x$  in Corollary \ref{Ozawa} (i) (resp.  (ii)) is called BMO (resp. VMO). It characterises the class of $x$ such that the commutator $[H,x]$ is bounded (resp.  compact).
\end{rem}
\begin{rem} The content of this remark is communicated to the authors by  N. Ozawa. Let $\M$ be a finite von Neumann algebra with a finite normal faithful trace $\tau$. Let $L^p({\M}), 1\leq p<\infty$ be the associated non commutative $L^p$ spaces (see \cite{PX03}). Recall that we set  $L^\infty(\M)=\M$.
For the operators $X\in B(L^2({\M})), p\geq 2$, define a semi-norm 
$$\|X\|_{L^p\rightarrow L^2}=\sup\{ \|Xy\|_{L^2({\M})};  y\in L^p{(\M})\subset L^2({\M}), \|y\|_{L^p({\M})}\leq 1\}.$$
Note $\|X\|_{L^2\rightarrow L^2}$ is just the operator norm $\|X\|$. Identify ${\M}$   as sub algebra of $B(L^2({\M}))$ by the left  multiplication on $L^2({\M})$. Let ${\M}'\subset B(L^2({\M}))$ be the sub algebra of the right multiplication of ${\M}$ on $L^2({\M}).$ For $b\in {\M}\cup {\M}'$, we have by H\"older's inequality that
$$\|b\|_{L^p\rightarrow L^2}=\|b\|_{L^q}$$
for $\frac1q+\frac1p=\frac12$.
The lemma  of \cite{O10}  Section 3 says that, for $X\in B(L^2({\M}))$,
\begin{eqnarray}
\label{O10lemma}  
\|X\|_{L^\infty\rightarrow L^2}\leq \inf\{\|Y\|\|b\|_{L^2(\M)}+\|Z\|\|c\|_{L^2(\M)}\}\leq 4\|X\|_{L^\infty\rightarrow L^2}.
\end{eqnarray}
Here the infimum is taken over all possible decomposition $X=Yb+Zc'$ with $Y,Z\in B(L^2({\M})), b\in {\M}, c'\in  {\M}', \|b\|,\|c\|\leq \|X\|$.
One can easily see that an analogue of the first inequality of (\ref{O10lemma}) holds for all $p>2$, that is
\begin{eqnarray}
\label{O10lemmap}  
\|X\|_{L^p\rightarrow L^2}\leq \inf\{\|Y\|\|b\|_{L^q(\M)}+\|Z\|\|c\|_{L^q(\M)}\},
\end{eqnarray}
for $\frac1q+\frac1p=\frac12$.
  Since $\|b\|^q_{L^q}\leq \|b\|^2_{L^2}\|b\|^{q-2}$, We get  the following H\"older-type inequality   for $X\in B(L^2({\M}))$,
\begin{eqnarray}
\label{Holder}
\|X\|_{L^p\rightarrow L^2}\leq 4\|X\|^{\frac{p-2}p}_{L^\infty\rightarrow L^2}\|X\|^{\frac2{p}}.
\end{eqnarray}
Suppose $Y\in B(L^2({\M}))$ satisfies that, for some $p>2$,
$$\|Y\|_{L^\infty \rightarrow L^p }=\sup \{ \|Yx\|_{L^p({\M})};  x\in L^\infty{(\M})\subset L^2({\M}), \|x\|_{\M}\leq 1\}<\infty.$$
Inequality (\ref{Holder}) implies that 
\begin{eqnarray}
\label{Holder2}\|XY\|_{L^\infty\rightarrow L^2}\leq \|X\|_{L^p\rightarrow L^2}\|Y\|_{L^\infty \rightarrow L^p }\leq 4\|X\|^{\frac{p-2}p}_{L^\infty\rightarrow L^2}\|X\|^{\frac2{p}}\|Y\|_{L^\infty \rightarrow L^p }.\end{eqnarray} 

Let ${\Bbb K}^L_{\cal M}\in B(L^2({\M}))$ be the collection of all operators sending the unit ball of ${\M}$ into a compact subset of $L^2({\M})$. Let ${\Bbb K}_{\M}=({\Bbb K}^L_{\M})^*\cap {\Bbb K}^L_{\M}$ be the associated $C^*$-algebra. Let $M({\Bbb K}_{\M})$ be the multiplier algebra of ${\Bbb K}_{\M}$, i.e. the algebra of all operators $X\in B(L^2({\M}))$ such that both $X{\Bbb K}_{\M}$ and ${\Bbb K}_{\M}X$ still belong to ${\Bbb K}_{\M}$.  Proposition of \cite {O10} Section 2 says that $X\in {\Bbb K}_{\M}$ iff for every sequence of finite rank projections $Q_n$ strongly converging to the identity of $B(L^2(\M))$, $\|X-Q_nX\|_{L^\infty\to L_2}\rightarrow 0$. Combining this with (\ref{Holder2}), we see that  $Y$ above belongs to $M({\Bbb K}_{\M})$. This applies to the particular case when $Y$ is the free Hilbert transform $H_\eps$ or $H^{op}_\eps$ and ${\M}$ is an amalgamated free product.
Ozawa suggested to study the $C^*$-algebra 
$$B_{\M}=\{X\in M({\Bbb K}_{\M}); [X,y]\in {\Bbb K}_{\M}, \forall y\in {\M}\subset B(L^2(\M))\}.$$
 Theorem \ref{main}  and Corollary \ref{Ozawa} (iii) imply that $H_\eps^{Rd} \in B_{{\cal L}(\F_n)}$ and similarly $H_\eps^{Ld} \in B_{{\cal L}'(\F_n)}$. Here ${\cal L}'(\F_n)$ is the   von Neumann algebra generated by the right regular representation $\rho_g$'s.

Let $\bar\F_n=\F_n\cup \partial \F_n$ and $C(\bar \F_n)$ be the $C^*$-algebra of continuous functions on $\bar\F_n$.  Note that $C(\bar \F_n)$ is isomorphic to the sub $C^*$-algebra of $ B(\ell^2(\F_n))$ generated by $\rho_g L_h\rho_{g^{-1}},   g,h\in\F_n$.  We then obtain 
$$C(\bar \F_n)\subset B_{{\cal L}'(\F_n)}.$$
\end{rem}

\subsection{Connections to Carr\'e du Champ }

We use the same notation to denote elements of $\F_\infty$ and points on its Cayley graph. The {\it Gromov product} for $g^{-1},g'$
(on the Cayley graph) is defined as
$$\langle g,g'\rangle=\frac {|g|+|g'|-|gg'|} 2.$$
A closely related object is  the so-called Carr\'e du Champ  of P. A. Meyer
\begin{eqnarray*}
\Gamma(\la_g,\la_{g'})&=&\frac{A(\la^*_g)\la_{g'}+\la^*_gA(\la_{g'})-A(\la_g^*\la_{g'})}2=\langle g^{-1},g'\rangle \la_{g^{-1}g'}
\end{eqnarray*}
associated to the conditionally negative operator $A:\la_g\mapsto |g|\la_g$.

The following is a key connection to the operator $L_h$ studied in previous sub sections, that
\begin{eqnarray}
\label{gromov}
2\Gamma(\la_g,\la_{g'})&=&\sum_{h\in \F_\infty} (L_h(\la_g))^*L_h(\la_{g'}).
\end{eqnarray}
Let us extend these notations to   $x=\sum_g c_g\la_g\in L^2(\hat\F_\infty)\otimes L^2(\m M)$, and set
\begin{eqnarray*}
A^r(x)&=&\sum_g c_g|g|^r\la_g\\
\Gamma(x,x)&=&\langle x, x\rangle= \sum c_g^*c_{g'}\langle g^{-1},g' \rangle \la_{g^{-1}g'}.
\end{eqnarray*}
We then have 
\begin{eqnarray}
\label{carre}
2\langle H_\varepsilon x, H_\varepsilon x\rangle=\sum_{h\in \F_\infty} |L_hx|^2=A(x^*)x+x^*A(x)-A(|x|^2).
\end{eqnarray} 

The following square function estimate was proved in \cite{JMP16}. One direction of the inequality had been proved in \cite{JM10} and \cite{JM12} in a more general setting.
\begin{lemma}(\cite{JMP16} Theorem A1, Example (c))\label{JMP}
For any $2\leq p<\infty$, $x\in L^p(\hat\F_\infty)\otimes L^p(\m M)$,
$$\|A^{\frac12} x\|_p\simeq^{\frac {p^4}{(p-1)^2}} \|(\sum_{h\in \F_\infty} |L_hx|^2)^\frac12\|_p+\|(\sum_{h\in \F_\infty} |L_h (x^*)|^2)^\frac12\|_p.$$
\end{lemma}

\begin{rem}
The equivalence above may fail if one replace $L_h (x^*)$ by $(L_hx)^*$ on the right hand side. Corollary 4.9 of \cite{JM12} gives   constants   $\simeq p$ for the ``$\lesssim$" direction.
\end{rem}

\subsection{Littlewood-Paley inequalities  }

In the case of the free group we adapt the definition of the paraproducts  studied in Section 3.3. Assume
$x=\sum_gc_g\la_g\in L^p, y=\sum_hd_h\la_h \in L^q$.    We then find that 
$$x\ddag y=\sum_{g^{-1}\nleqslant h}c_gd_h\la_{gh},\ \ \  x\dag y=\sum_{g^{-1}< h}c_gd_h\la_{gh}.$$ 
Recall that we write $g\leq h$ (or $h\geq
g$) if $h=gk$ with $g,h,k$ reduced words and $g<h$ if $g\leq h$ and $g\neq h$.



We consider a decomposition of $\F_\infty$ into disjoint geodesic paths.
To get one, first pick a (randomly decided) geodesic path ${\mathbb P}_0$ starting
at the unit element $e$. Then  for any  length $1$ elements not in ${\mathbb P}_0$
pick a (randomly decided) geodesic path starting at each of them. We
then go to length $2$ elements which are not contained in any of the
previous picked paths, and pick a (randomly decided) geodesic path
starting at each of them.  We repeat this procedure and get countable
many disjoint geodesic paths ${\mathbb P}_n$ such that $\cup_n
{\mathbb P}_n=\F_\infty$.

Let $T_n$ be the $L^2$-projection onto the span of ${\mathbb P}_n$. Let
$h_1(n)$ be the root of ${\mathbb P}_n$, i.e. the first element in ${\mathbb
  P}_n$. 
  Let $S_n$ be the projection to the collection of words
smaller than $h_1(n)$ (note that $S_0=0$).


\begin{corollary}\label{branch}
For any $1<p<\infty$, the maps
 $T_n$ are completely bounded on $L^p$ with  
\begin{eqnarray}
\label{Tn} 
\|T_n\|_{p\to p}\lesssim {c_p^2}.
\end{eqnarray}
Moreover, for any $p>2$
\begin{eqnarray}
\label{sumTn}
\|\sum_n |T_nx+S_nx|^2-|S_nx|^2\|_{\frac p2}\lesssim  c^2_p\|x\|^2_p.
\end{eqnarray} 
\end{corollary}

\begin{proof}
We write $x=\sum c_g\la_g$ and $T_nx=\sum _{g\in {\mathbb P}_n} c_g\la_g$.
Then
\begin{eqnarray}
(T_nx)^*T_nx-\sum_{g\in {\mathbb P}_n} |c_g|^2\la_e
&=&\sum_{g< h \in {\mathbb P}_n}c^*_gc_h\la_{g^{-1}h}+\sum_{h< g \in {\mathbb P}_n}c^*_gc_h\la_{g^{-1}h}\nonumber\\
&=& (T_nx)^*\dag  T_nx+(( T_nx)^* \dag T_nx)^*.\nonumber
\end{eqnarray}
Since $(T_nx+S_nx)^*\dag T_nx=x^*\dag T_nx$, we have that 
$$(T_nx)^*\dag  T_nx =x^*\dag T_nx-(S_nx)^*T_nx.$$
Therefore,
\begin{eqnarray}\label{idTn}
(T_nx)^*T_nx-\sum_{g\in {\mathbb P}_n} |c_g|^2\la_e=x^*\dag  T_nx+(x^* \dag T_nx)^*-(S_nx)^*T_nx-(T_nx)^*S_nx.
\end{eqnarray}

In particular for $n=0$, we have actually
$$(T_0x)^*T_0x-\sum_{g\in {\mathbb P}_0} |c_g|^2\la_e= x^*\dag T_0x +( x^*\dag T_0x)^*.$$
  Assume $p>2$, by Proposition \ref{paraid}, we have  
 $$\|T_0(x)\|_{p}^2 \leq (4+2c_pc_{\frac p
    2})\|x\|_{p}\|T_0x\|_{p}+\|x\|^2_p.$$ So $\|T_0(x)\|_{p}\leq
  (5+2c_{\frac p2}c_p) \|x\|_p$ for $p>2$. One concludes that $T_0$ is (completely) bounded 
on $L^p$. One can improve the bound on $\|T_0\|_{p\to p}$
when $p$ is close to 2 by using interpolation.  The case $p<2$ follows by duality.
Thus we have obtained (\ref{Tn}) for an arbitrary ${\mathbb P}_0$ starting at $e$, for general ${\mathbb P}_n$ this follows by using translations.

Summing (\ref{idTn}) over $n$, we get
\begin{eqnarray*}
\sum_{n\geq 0} |T_nx+S_nx|^2-|S_nx|^2&=&\sum_{n\geq 0}[(T_nx)^*T_nx+(S_nx)^*T_n(x)+(T_nx)^*S_nx]\\
&=&\sum_{n\geq 0}x^*\dag   T_nx +( x^*\dag  T_nx)^*+
\sum_g |c_g|^2\la_e\\
&=&  x^*\dag x +(x^*\dag x)^*+
\tau |x|^2\la_e .
\end{eqnarray*}
(\ref{sumTn}) then follows from Proposition \ref{paraid}. 
\end{proof}
We now consider a concrete partition given by geodesic paths. For any
$h_0\notin \m R_{g^{\pm}}$ and $g\in S$, let $\mathbb
P_{h_0,g}=\{h_0g^k;k\in \N\}$, they form a countable partition of
$\F_\infty\setminus\{e\}$, we may index it with $\Z^*=\Z\setminus\{0\}$. We still denote the root of 
$\mathbb P_n$ by  $h_1(n)$ and put $h_0(n)=h_0$, $k_n=k$ if $\mathbb P_n=\mathbb
P_{h_0,g_k}$ . By definition $h_0(n)\in \m R_{g_{k_n}^\pm}^\bot$ if $h_1(n)\in \m R_{g_{k_n}}^\pm$.

\begin{lemma} \label{new} Let $T_n$ be the $L^2$-projection onto   $ {\mathbb P}_n$ described above, we have for any $p\geq 2$, $x\in L^p$.
\begin{eqnarray}
\label{sumdotTn}
\|(\sum_{n\in \Z^*} | T_nx |^2)^\frac  12 \|_{L^p}\lesssim c_p   \|x\|_p.
\end{eqnarray}
\end{lemma}
\begin{proof} 


We may assume $\tau x=0$. Let $E_k$ be the projection from the group
von Neumann algebra ${\cal L}(\F_\infty)$ onto the von Neumann algebra
generated by $\la_{g_k}$. We can easily verify that for $k\in \Z^*$
\begin{eqnarray*}
 E_{|k|} | R_{g_k}x|^2 
 &=&E_{|k|}|\sum_{h_1(n)\in R_{g_k}} T_nx  |^2  =\sum_{h_1(n)\in R_{g_k}} | T_nx|^2,
 \end{eqnarray*}
  because, if  $h_1(n),h_1(n')\in \m R_{g_k}$, then $h_0(n),h_0(n')\in \m R_{g_k}^\bot$ and $h^{-1}_0(n)h_0(n')\in E_{|k|}(\F_\infty)$ iff $n=n'$. 
 Therefore,
\begin{eqnarray*}
\sum_{n\in {\Z^*}} | T_nx  |^2  =\sum_{k=1}^\infty E_k (| R_{g_k}x|^2+| R_{g{-k}}x|^2)= \tau |x|^2 + \sum_{k=1}^\infty \bmr{E_k (| R_{g_k}x|^2+| R_{g_{-k}}x|^2)}.
\end{eqnarray*}
By the free Rosenthal inequality (Theorem A in \cite{JPX07}) for length one polynomials, we get
for $p\geq 4$, with $X_k=\bmr{E_k (| R_{g_k}x|^2+|
  R_{g_{-k}}x|^2)}$.
\begin{eqnarray*}
\|(\sum_n |T_nx |^2)^\frac 12 \|^2_{p}&\lesssim & \tau |x|^2+
(\sum_{k\in \N}\|X_k\|_{\frac p2}^\frac p2)^\frac2p +(\sum_{k\in \N}\| X_k \|
_{2}^2)^\frac12\\ &\lesssim& (\sum_{k\in \Z^*}\| R_{g_k}x \| _{
  p}^p)^\frac 2p +(\sum_{k\in \Z^*}\| R_{g_k}x \|
_{4}^4)^\frac12\\ &\lesssim& \|(\sum_{k\in \Z^*}| R_{g_k}x |^2 )^\frac1
2\|_p^2 \lesssim^{c_p^2} \|x\|_p^2.
\end{eqnarray*}
Where we used the obvious facts by interpolation that
$L^p(\ell_2^c)\to \ell^p(L^p)$ and $L^p(\ell_2^c)\to \ell_4(L_4)$ are
contractions.
 The case of $p=2$ is
obvious. We then get the estimate for all $2\leq p<\infty$ by
interpolation.
 \end{proof} 
\bigskip
Let ${\mathbb P_j}=\{h_{1}(j)<h_{2}(j)<\cdots h_{k}(j)< \cdots\}$ be arbitrary geodesic paths of $\F_\infty$. For $x_j=\sum_{k\in \N} c_{k}\la_{h_{k}(j)}$ supported on $\mathbb P_j$, we consider its dyadic parts 
\begin{equation}M_{n,j} x=\sum_{2^{n}\leq k<2^{n+1}}  c_{k} \la_{h_{k}(j)}.
\end{equation}
 Dealing with $\F_1=\Z$ with the $\N\cap\{0\}$ and $-\N$ as geodesic paths, the classical Littlewood-Paley theory says that 
\begin{equation}\label{lp1}
\|(\sum_{n=1}^\infty|M_{n,1}x|^2+|M_{n,2}x|^2)^\frac12\|_{L^p}\simeq ^{c_p}\|x\|_{L^p}\end{equation}
for all $1<p<\infty$ and   $x\in C_c(\Z)$.

In \cite{JMP16} , the authors proved a ``smooth" one sided version of  (\ref{lp1}) for  $p>2$ and for $x$ supported on a  geodesic path. The following theorem is  a ``truncated" version of it and says a little more.  

\begin{thm} \label{lpp} For $x_j$ supported on geodesic paths ${\mathbb P}_{j}$, we have
\begin{equation}
\|(\sum_{n,j=1}^\infty|M_{n,j}x_j|^2)^\frac12\|_{p}\leq C{p^2c_p^2}  \|(\sum_j|x_j|^2)^\frac12\|_{p}\end{equation}
for all $2\leq p<\infty$.
\end{thm}

\begin{proof}
As usual $g_1,g_2...$ are the free generators of $\F_\infty$. We embed
$\F_\infty$ into the free product $\F_\infty*\F_\infty$ and denote by
$g'_1,g'_2,...$ the generators of the second copy of $\F_\infty$. Let
$y_j=\la_{g'_jh^{-1}_1(j)}x_j$. The $y_j$'s are supported on disjoint paths
${\mathbb P}'_j\subset \F_\infty*\F_\infty$ with roots of distinct
generators $g'_j$.  Note $|x_j|^2=|y_j|^2$ and $|M_{n } x_j|^2=|M_n
y_j|^2$. By considering $y_j$ instead, we may assume ${\mathbb
  P}_{j}=\{h_{1}(j)<h_{2 }(j)<\cdots h_{k}(j)\cdots\}$ with
$|h_k(j)|=k$ and $L_{h_k(j)}x_m=0$ for $j\neq m$.

Let \begin{eqnarray*}
M_{\varphi_n,j}=  {2^{1-\frac n2}}\sum _{2^{n-1}< k\leq 2^n}&& L_{h_{k}(j)}{A^{-\frac12}}+ \sum _{2^{n}< k\leq 2^{n+1}}(\sqrt{k}-\sqrt{ k-1}) L_{h_{k}(j)}{A^{-\frac12}}\\
&-& {2^{-\frac{n+1}2}}\sum _{2^{n+1}< k\leq 2^{n+2}} L_{h_{k }(j)}{A^{-\frac12}}.
\end{eqnarray*}
Then  
$M_{\varphi_n,j}(\la_{h_{l}(m)})=0$ unless $m=j$ and $l\in (2^{n-1},2^{n+2}]$, 
and one can check that
$$M_{\varphi_n,j}(\la_{h_{l }(j)})={\varphi_n}(l)\la_{h_{l}(j)},$$
for some $\varphi_n:\N\to \R$ with $\chi_{[2^{n}, 2^{n+1})}\leq \varphi_n\leq \chi_{(2^{n-1}, 2^{n+2})}$.
Note $$M_{\varphi(n),j}=\sum _{2^{n-1}< k\leq2^{n+2}} a_{k,j}L_{h_{k}(j)} A^{-\frac12}$$ 
with $\sum_ka^2_{k,j}\leq c$. By the convexity of the operator valued function $  |\cdot |^2$, we have 
$$|M_{\varphi_n,j}x_j|^2\leq c\sum_{2^{n-1}<k\leq2^{n+2}}|L_{h_{k}(j)}A^{-\frac12}x_j|^2,$$
and $M_{\varphi_n,j}x_m=0$ for $m\neq j$. 
Note $M_{\varphi_n,j}, M_{\varphi_{n'},j}$'s are disjoint for $|n-n'|\geq 2$. Applying Lemma \ref{JMP} to $x=\sum_j x_j\otimes c_j$, we obtain,
\begin{eqnarray*}
\|(\sum_{n=0}^\infty|M_{\varphi_n,j}x_j|^2)^\frac12\|_{p}\leq c\|( \sum_{k,j}|L_{h_{k}(j)}A^{-\frac12}x_j|^2)^\frac12\|_p \leq cp^2  \|(\sum_j|x_j|^2)^\frac12\|_p.
\end{eqnarray*}

Assume $h_{2^n}(j)\in \m R_{g_{n,j}}, h_{2^{n+1}}(j)\in \m R_{g_{n',j}}$, we have that 
\begin{eqnarray}
\la_{h_{2^{n+1}(j) }}L_{g_{n',j}^{-1}}\la_{h_{2^{n+1}}(j)^{-1}}(\la_{h_{2^n-1}(j)}L_{g_n,j}\la_{h_{2^n-1}(j)^{-1}}M_{\varphi_n}x_j)=M_{n}x_j,
\end{eqnarray}
because $h_{2^{n+1}+1}(j)\in \m R^\bot_{g^{-1}_{n',j}}.$
By (\ref{klcf}),
\begin{eqnarray*}
\|(\sum_{n=1}^\infty|M_{n,j}x_j|^2)^\frac12\|_{p}\leq cc_p^2\|(\sum_{n=1}^\infty| M_{\varphi_n ,j}x_j|^2)^\frac12\|_{p}\leq cp^2c_p^2   \|(\sum_j|x_j|^2)^\frac12\|_{p}
\end{eqnarray*}
for all $2\leq p<\infty$.
\end{proof}

Let $M_{n,k}x=M_{n,k}T_kx$ for the $T_k$ in Lemma \ref{new}. We obtain the following  from Theorem \ref{lpp} and duality: 
\begin{corollary}   \label{llpp}For all $2\leq p<\infty$, and $x\in L^p$ 
\begin{equation}
\max\Big\{\|(\sum_{n,k }|M_{n,k}x|^2)^\frac12\|_{p},\|(\sum_{n,k }|M_{n,k}(x^*)|^2)^\frac12\|_{p}\Big\}\leq Cp^2c_p^3  \|x\|_{p},\end{equation}
for all $1\leq p<2$, and $x\in L^p$  
\begin{equation*}
\|x\|_p\leq Cp'^2c_{p'}^3  \inf\Big\{ \|(\sum_{n,k }|M_{n,k}y|^2)^\frac12\|_{p}+\|(\sum_{n,k }|M_{n,k}(z^*)|^2)^\frac12\|_{p}; x=y+z\Big\}. 
\end{equation*}
\end{corollary}

 \bigskip
{\bf Acknowledgment.} The first author would like to  thank  Marius Junge  for helpful discussions. An initial argument for Theorem \ref{main2} was  obtained during a  visit to him at  Urbana-Champaign.   





\bibliographystyle{amsplain}

\bigskip
\hfill \noindent \textbf{Tao Mei} \\
\null \hfill Department of Mathematics
\\ \null \hfill Baylor University \\
\null \hfill One bear place, Waco, TX  USA \\
\null \hfill\texttt{tao\_mei@baylor.edu}

\bigskip
\hfill \noindent \textbf{\'Eric Ricard} \\
  \null \hfill Laboratoire de Mathematiques Nicolas Oresme 
 \\ \null \hfill   Normandie Univ, UNICAEN, CNRS\\ \null \hfill  14032 Caen FRANCE\\
\null \hfill\texttt{eric.ricard@unicaen.fr}
\end{document}